\documentclass[a4paper,10pt]{amsart}
\usepackage{amsmath,amssymb,amsthm,mathtools}
\usepackage[unicode,colorlinks=true,linkcolor=blue,citecolor=blue,urlcolor=blue]{hyperref}
\usepackage{microtype}

\theoremstyle{plain}
\newtheorem{theorem}{Theorem}

\newtheorem{proposition}{Proposition}

\theoremstyle{definition}

\newcommand{\ZZ}{\mathbb{Z}}
\newcommand{\RR}{\mathbb{R}}

\newcommand{\Ci}{\operatorname{Ci}}

\title{A half-shift reflection identity for the digamma function}
\author{Nikita Kalinin}

\begin{document}

\maketitle

\begin{abstract}
We prove the identity

\[
2W_1(x) + \log 4 + \psi\left(\tfrac{1}{2} + x\right) + \psi\left(\tfrac{3}{2} - x\right) = 0,
\] 
where \(\psi\) is the digamma function and 
\[
W_1(x) = 2\int_0^\infty \Re\left( \frac{y}{(y^2+1)(e^{\pi(y+2ix)} - 1)} \right) dy.
\]
The identity was first conjectured while studying class number $h(D)$ for $D=m^2$ from two complementary perspectives.
Our proof, however, is purely analytic: we compute cosine-series expansions of both sides, expressed in terms of the cosine integral \(\Ci(z)\). 

Using the above identity and M\"obius inversion we find an elementary formula for $$\sum_{\substack{1\le r<m\\ (r,m)=1}} W_1\!\left(\frac{r}{m}\right).$$ 

\noindent\textbf{Keywords:} digamma (polygamma) function; cosine integral (Ci); exponential integral (E1); Fourier cosine series; Möbius inversion

\noindent\textbf{AMS Classification:} 33B15, 33B10, 42A16, 44A10, 11A25
\end{abstract}

\section{Introduction}

The digamma function \(\psi(z)\), defined as the logarithmic derivative of the gamma function \(\Gamma(z)\), appears frequently in number theory and analysis. Identities involving \(\psi\) (see \cite[pp. 952–955]{gradshteyn2014table}) often encode deep arithmetic or analytic properties related to Dirichlet \(L\)-functions or Euler sums \cite{dixit2010laplace}, \cite{bradley1996ramanujan}.

Define
\[
W_1(x) = 2\int_0^\infty \Re\left( \frac{y}{(y^2+1)(e^{\pi(y+2ix)} - 1)} \right) dy.
\]

It is immediate that $W_1(x)$ is $1$-periodic and even. Our main result is  
\begin{theorem}\label{thm:main}
For each real $x$ with $\tfrac12\pm x\notin\{0,-1,-2,\dots\}$, 
\begin{equation}\label{eq:main-id}
2W_1(x) + \log 4 + \psi\left(\tfrac{1}{2} + x\right) + \psi\left(\tfrac{3}{2} - x\right) = 0.
\end{equation}  
\end{theorem}

In particular, \eqref{eq:main-id} holds for all rational numbers \(x = r/m \in (0,1)\).

The function \(W_1(x)\) arises naturally in the context of topographs for binary quadratic forms and is a part of Theorem 9.4 and formula (9.29) about the class number formula for a square discriminant in \cite{o2024topographs}. Using other approach to compute the same number via telescoping the author expressed $2W_1(x)$ via $ \psi\left(\tfrac{1}{2} + x\right) + \psi\left(\tfrac{3}{2} - x\right)$ plus some other terms \cite{kalinintopo}. While computing examples, it was observed that all the other terms collapse to $\log 4$ without any apparent reason. However the identity \eqref{eq:main-id} can be verified numerically. It seems difficult to prove the above identity using topographs, so we prove it independently, by purely analytic methods.
 
The integral representation of \(W_1\) suggests a connection with Laplace transforms and periodic functions, which we exploit here via Fourier analysis.

Our method consists of expanding both sides of \eqref{eq:main-id} in Fourier cosine series. The digamma sum is expanded using classical identities, while \(W_1(x)\) is expanded term-wise via integration techniques involving the exponential integral \(E_1(z)\) and the cosine integral \(\Ci(z)\). The matching of coefficients yields the result.

This identity is reminiscent of known formulas for digamma functions at rational arguments, but the appearance of the integral \(W_1(x)\) and its explicit Fourier expansion seem to be new. The result may also be interpreted as a kind of functional equation for a certain Laplace-type transform of the periodic digamma function.

We also compute the reduced--residue sum of $W_1$ appearing in the class--number identity~(9.29) of \cite{o2024topographs}.

\begin{proposition}\label{prop:main}
Let \(m>1\). Then
\begin{equation}\label{eq:SW1-closed}
\sum_{\substack{1\le r<m\\ (r,m)=1}} W_1\!\left(\frac{r}{m}\right)
\;=\; \varphi(m)\log\!\Big(\frac{m}{2}\Big)
\;+\; \varphi(m)\sum_{p\mid m}\frac{\log p}{p-1}
\;-\; m\sum_{d\mid m}\frac{\mu(d)}{d}\,\psi\!\left(\frac{m}{2d}\right)\, .
\end{equation}
\end{proposition}
Here \(\varphi\) denotes Euler’s totient function and \(\mu\) is the Möbius function; the sum \(\sum_{p\mid m}\) runs over primes dividing \(m\).

In view of \eqref{eq:main-id}, it would be good to further simplify $W_1(\frac{r}{m})+W_1(\frac{s}{m})$ for $1\leq r,s,<m, rs\equiv 1 \pmod m$, since such expression appear in \cite{o2024topographs} but it seems that it is beyong our current understanding of sums like $\psi(\frac{r}{m})+\psi(\frac{s}{m})$. We mention what is known: $\psi(r/m)+\gamma$ is transcendental \cite{murty2007transcendental}, and we have Gauss' formula \cite{gaub1812disquisitiones},\cite{hashimoto2008gauss},\cite{kolbig1996polygamma}:

$$\psi(\frac{r}{m})=-\gamma-2m-\frac{\pi}{2}\cot(\frac{\pi r}{m})+2\sum_{j=1}^{m/2}\left(\cos\frac{2\pi rj}{m}\right)\log\sin\frac{\pi j}{m},$$

and also

$$\psi(\frac{r}{m})=-\gamma - \log(m)+\sum_{j=1}^{m-1}e^{-2\pi i j r/m}\log(1-e^{2\pi i j/m}).$$

The article \cite{o2024topographs} also mentions a function 
\[
W_2(x) = 2\int_0^\infty \Re\left( \frac{y(3y^4+5y^2+6)}{(y^2+1)(e^{\pi(y+2ix)} - 1)} \right) dy.
\]

So far we do not know how to relate it to any other function since it is not clear how to employ the second approach (with telescoping) in this case.

{\bf Outline. } In Section~\ref{cos} we find the cosine Fourier expansion of both sides in \eqref{eq:main-id}. In Section~\ref{proofs} we prove Theorem~\ref{thm:main} and Proposition~\ref{prop:main}.

\section{Cosine series}\label{cos}

We begin by deriving a cosine series for $f(x)=\psi(\tfrac12+x)+\psi(\tfrac32-x)$. Set
\[
f(x)=\psi(1+u)+\psi(1-u), 0<x<1, u=x-\tfrac12\in(-\tfrac12,\tfrac12)
\]
Then
\[
\int_0^1 f(x)\,dx
=\int_{-1/2}^{1/2}\bigl(\psi(1+u)+\psi(1-u)\bigr)\,du
=2\log\!\left(\frac{\Gamma(3/2)}{\Gamma(1/2)}\right)=-2\log 2,
\]
hence the constant Fourier coefficient is $a_0/2=-2\log 2$.

Since $f$ is an even function, we need to find

\[
a_k
= 2\!\int_{0}^{1}\!f(x)\cos(2\pi kx)\,dx
= 2(-1)^k \!\!\int_{-1/2}^{1/2}(\psi(1+ u)+\psi(1- u))\cos(2\pi ku)\,du .
\]

Using the classical expansion
\[
\psi(1+ u)=-\gamma+ \sum_{n\ge1}\left(\frac{1}{n}-\frac{1}{n+u}\right)
\]
we obtain
\[
\psi(1+ u)+\psi(1- u)= -2\gamma -
      \sum_{n=1}^{\infty}\left(\frac{2}{n}-\frac{1}{n-u}-\frac{1}{n+u}\right).
\]

We obtain
\[
\int_{-1/2}^{1/2}\cos(2\pi ku)\left(
  \frac{1}{n-u}+\frac{1}{n+u}\right)\,du
= 2
  \int_{\,n-1/2}^{\,n+1/2}\!
       \frac{\cos(2\pi k t)}{t}\,dt
\quad (t:=n\pm u).
\]

Hence
\[
a_k
= -4(-1)^k
  \int_{1/2}^{\infty}
        \frac{\cos(2\pi k t)}{t}\,dt .
\]

Let \(s=2\pi k t\).  Then \(dt/t = ds/s\):

\[
a_k
= -4(-1)^k
  \int_{\pi k}^{\infty}
        \frac{\cos s}{s}\,ds
= 4(-1)^k\,\Ci(\pi k),
\qquad k\ge 1,
\]
where
\[
\Ci(z) := -\!\int_{z}^{\infty}\frac{\cos t}{t}\,dt .
\]

Because \(f\) is even about \(x=\tfrac12\), all sine-coefficients vanish.  
Collecting the constant term and the cosine coefficients:

\[
\psi\!\Bigl(\tfrac12+x\Bigr)+\psi\!\Bigl(\tfrac32-x\Bigr)
\;=\;
-2\log 2
\;+\;
4\sum_{k=1}^{\infty}
  (-1)^k\,\Ci(\pi k)\,
  \cos(2\pi k x)
,\qquad 0<x<1 .
\]

\subsection{Cosine series for $W_1(x)$}

Let us write a cosine series for $W_1(x)$. Expanding

\[
(e^{\pi(y+2ix)}-1)^{-1}=\sum_{n\ge1}e^{-\pi n y}e^{-2\pi i n x}
\]
for $y>0$ and taking real parts yields
\[
W_1(x)=\sum_{k=1}^\infty a_k\,\cos(2\pi k x),
\qquad
a_k:=2\int_0^\infty \frac{y}{y^2+1}\,e^{-\pi k y}\,dy.
\]

Let us evaluate the cosine coefficients. Write
\(
\frac{y}{y^2+1}=\frac12\bigl(\frac{1}{y-i}+\frac{1}{y+i}\bigr)
\).
For $a>0$ the Laplace–Euler integral gives
\[
\int_0^\infty \frac{e^{-a y}}{y\pm i}\,dy
=e^{\pm i a}\,E_1(\pm i a),
\qquad
E_1(z):=\int_z^\infty \frac{e^{-t}}{t}\,dt,
\]
so with $a=\pi k$,
\[
a_k
=2\Re\!\Bigl(e^{i\pi k}E_1(i\pi k)\Bigr).
\]
For $x>0$ one has
\[
E_1(ix)=-\Ci(x)\;-\;i\bigl(\operatorname{Si}(x)-\tfrac{\pi}{2}\bigr)
\quad\Rightarrow\quad
\Re\,E_1(ix)=-\Ci(x),
\]
whence
\begin{equation}\label{eq:ak}
\quad a_k=2\,(-1)^{k+1}\,\Ci(\pi k),\qquad k\ge1.
\end{equation}
Since $\Ci(\pi k)=O(k^{-2})$, the cosine series for $W_1$ and for $f$ converge absolutely and uniformly on $[0,1]$.

\section{Proofs}\label{proofs}

\begin{proof}[Proof of Theorem~\ref{thm:main}]
By \eqref{eq:ak} and \S2 we have, for $0<x<1$,
\[
2\,W_1(x)
=2\sum_{k\ge1} a_k\cos(2\pi kx)
=4\sum_{k\ge1}(-1)^{k+1}\Ci(\pi k)\cos(2\pi kx),
\]
while
\[
-\log 4 - f(x)
= -\log 4 + 2\log 2
-4\sum_{k\ge1}(-1)^k\Ci(\pi k)\cos(2\pi kx)=\]
\[=4\sum_{k\ge1}(-1)^{k+1}\Ci(\pi k)\cos(2\pi kx).
\]
Thus $2W_1(x)=-\log 4 - f(x)$ for $0<x<1$. Both sides extend to a continuous $1$-periodic function on $\mathbb{R}$ (away from the poles of $\psi$), hence the identity holds for all such $x$, proving the theorem.
\end{proof}
\medskip

\begin{proof}[Proof of Proposition~\ref{prop:main}] We closely follow Lehmer’s approach \cite{lehmer1975euler} below.

We will repeatedly use the Gauss multiplication theorem for \(\psi\):
\begin{equation}\label{eq:psi-mult}
\sum_{j=0}^{m-1} \psi\!\left(z+\frac{j}{m}\right) \;=\; m\,\psi(mz) - m\log m,
\qquad z\in\RR,\quad m\in\ZZ_{\ge 1}.
\end{equation}

Because the map \(r\mapsto m-r\) permutes the reduced residue system modulo \(m\), we have
\begin{equation}\label{eq:symmetry}
\sum_{\substack{1\le r<m\\(r,m)=1}}\psi\!\left(\tfrac32-\frac{r}{m}\right)
\;=\;
\sum_{\substack{1\le r<m\\(r,m)=1}}\psi\!\left(\tfrac12+\frac{r}{m}\right).
\end{equation}
Hence, from \eqref{eq:symmetry},
\begin{equation}\label{eq:SW1-step1}
\sum_{\substack{1\le r<m\\ (r,m)=1}} W_1\!\left(\frac{r}{m}\right)
= -\frac12\Big(\varphi(m)\log 4 + 2\!\!\sum_{\substack{1\le r<m\\(r,m)=1}}\psi\!\left(\tfrac12+\frac{r}{m}\right)\!\Big)
= -\varphi(m)\log 2 - \sum_{\substack{1\le r<m\\(r,m)=1}}\psi\!\left(\tfrac12+\frac{r}{m}\right).
\end{equation}

Define the ``full'' (non-primitive) sum
\begin{equation}\label{eq:T12-def}
T(m) \;:=\; \sum_{j=1}^{m-1}\psi\!\left(\tfrac12+\frac{j}{m}\right).
\end{equation}
From \eqref{eq:psi-mult} with \(z=\tfrac12\) and subtracting the \(j=0\) term, we get
\begin{equation}\label{eq:T12-eval}
T(m)
\;=\; m\,\psi\!\left(\frac{m}{2}\right) - m\log m - \psi\!\left(\tfrac12\right).
\end{equation}

Group the terms in \eqref{eq:T12-def} by the common divisor of \(j\) and \(m\): for each divisor \(n\mid m\) with \(n>1\) we set
\begin{equation}\label{eq:S12star-def}
T^*(n) \;:=\; \sum_{\substack{1\le r<n\\(r,n)=1}}\psi\!\left(\tfrac12+\frac{r}{n}\right),
\end{equation}
so that
\begin{equation}\label{eq:T12-group}
T(m) \;=\; \sum_{\substack{n\mid m\\ n>1}} T^*(n).
\end{equation}
By M\"obius inversion (and using \(\sum_{d\mid m}\mu(d)=0\) for \(m>1\)) we obtain
\begin{equation}\label{eq:S12star-mobius}
T^*(m)
\;=\; \sum_{d\mid m}\mu(d)\,T\!\Big(\frac{m}{d}\Big)
\;=\; m\sum_{d\mid m}\frac{\mu(d)}{d}\,\psi\!\left(\frac{m}{2d}\right)
\;-\; m\sum_{d\mid m}\frac{\mu(d)}{d}\log\!\Big(\frac{m}{d}\Big).
\end{equation}

Two standard divisor-sum identities (multiplicativity with a check on prime powers) are
\begin{equation}\label{eq:divisor-identities}
\sum_{d\mid m}\frac{\mu(d)}{d} \;=\; \frac{\varphi(m)}{m},
\qquad
\sum_{d\mid m}\frac{\mu(d)}{d}\log d \;=\; -\frac{\varphi(m)}{m}\sum_{p\mid m}\frac{\log p}{p-1},
\end{equation}
whence
\begin{equation}\label{eq:log-sum}
m\sum_{d\mid m}\frac{\mu(d)}{d}\log\!\Big(\frac{m}{d}\Big)
\;=\; \varphi(m)\log m + \varphi(m)\sum_{p\mid m}\frac{\log p}{p-1}.
\end{equation}
Combining \eqref{eq:SW1-step1}, \eqref{eq:S12star-mobius}, and \eqref{eq:log-sum} yields
\begin{align*}
\sum_{\substack{1\le r<m\\ (r,m)=1}} W_1\!\left(\frac{r}{m}\right)
&= -\varphi(m)\log 2 - T^*(m)\\
&= -\varphi(m)\log 2 - \Big[\,m\sum_{d\mid m}\frac{\mu(d)}{d}\,\psi\!\left(\frac{m}{2d}\right) - \varphi(m)\log m - \varphi(m)\sum_{p\mid m}\frac{\log p}{p-1}\,\Big]\\
&= \varphi(m)\log\!\Big(\frac{m}{2}\Big) + \varphi(m)\sum_{p\mid m}\frac{\log p}{p-1}
\;-\; m\sum_{d\mid m}\frac{\mu(d)}{d}\,\psi\!\left(\frac{m}{2d}\right),
\end{align*}
which is precisely \eqref{eq:SW1-closed}.
\end{proof}

\section{Data availability statement}
No data was used.
\bibliography{../bibliography.bib}
\bibliographystyle{abbrv}
\end{document}